\documentclass[11pt]{amsart}

\usepackage[centertags]{amsmath}
\usepackage{amsfonts}
\usepackage{amssymb}
\usepackage{mathrsfs}
\usepackage[english]{babel}
\usepackage[active]{srcltx}

\usepackage[colorlinks]{hyperref}

\newcommand{\dist}{\operatorname{dist}}

\newcommand{\Real}{\mathbb{R}}
\newcommand{\Natural}{\mathbb{N}}

\newcommand{\inte}{\operatorname{int}}

\newcommand{\Lip}{\operatorname{Lip}}
\newcommand{\supp}{\operatorname{supp}}

\newcommand{\LsXZ}{\mathcal{L}(X,Z)}
\newcommand{\LsYZ}{\mathcal{L}(Y,Z)}

\newtheorem{thm}{Theorem}[section]
\newtheorem{cor}[thm]{Corollary}
\newtheorem{lem}[thm]{Lemma}
\newtheorem{prop}[thm]{Proposition}

\newtheorem{rem}[thm]{Remark}
\newtheorem{ex}[thm]{Example}

\newtheorem{defn}[thm]{Definition}

\numberwithin{equation}{section}

\begin{document}

\title[Smooth extension of mappings]{On smooth extensions of vector-valued functions defined on closed subsets of Banach spaces}

\author{M. Jim{\'e}nez-Sevilla and L. S\'anchez-Gonz\'alez}

\address{Departamento de An{\'a}lisis Matem{\'a}tico\\ Facultad de
Matem{\'a}ticas\\ Universidad Complutense de Madrid\\ 28040 Madrid, Spain}

\thanks{Supported in part by DGES (Spain) Project MTM2009-07848. L. S\'anchez-Gonz\'alez has  also been supported by grant MEC AP2007-00868}

\email{marjim@mat.ucm.es,
lfsanche@mat.ucm.es}



\date{December, 2011}

\maketitle
\begin{center}{\emph{{\small To the memory of Robb Fry}}}\end{center}

\begin{abstract}
Let $X$ and $ Z$ be Banach spaces, $A$ a closed subset of $X$ and a mapping $f:A\to Z$. We give necessary and sufficient conditions to obtain a $C^1$ smooth mapping $F:X \to Z$ such that $F_{\mid_A}=f$, when either (i) $X$ and $Z$ are Hilbert spaces and $X$ is separable, or (ii)   $X^*$ is separable and $Z$  is an absolute Lipschitz retract, or (iii) $X=L_2$ and $Z=L_p$ with $1<p<2$, or (iv) $X=L_p$ and $Z=L_2$ with $2<p<\infty$.
\end{abstract}

\section{Introduction}

In this note, we study how  the tecniques  given in \cite{Azafrykeener} and \cite{MarLuis} can be applied  to obtain a $C^1$ smooth extension
of a $C^1$ smooth and {\em vector-valued} function defined on a closed subset of a Banach space. More precisely, if $X$ and $Z$ are Banach spaces, $A$  is a closed subset of $X$  and  $f:A\to Z$ is a mapping, under what conditions does there exist a $C^1$ smooth mapping $F:X\to Z$ such that the restriction of $F$ to $A$ is $f$?
This note is the second part of our previous paper \cite{MarLuis}, where we studied  the problem of the extension of  a $C^1$
smooth and {\em real-valued} function defined on a closed subset  of a non-separable Banach space $X$ to a  $C^1$ smooth function defined on $X$.

The notation we use is standard. In addition, we shall follow, whenever possible, the notations given in \cite{Azafrykeener}
and \cite{MarLuis}. We shall denote a norm in a Banach space $X$  by $||\cdot||_X$ (or  $||\cdot||$ if the Banach space $X$ is understood). We denote by $B_X(x,r)$ the open ball with center $x\in X$ and radius $r>0$ (or $B(x,r)$ if the Banach space $X$ is understood). We write the closed  ball as $\overline{B}_X(x,r)$ (or $\overline{B}(x,r)$). We denote by $\LsXZ$ the space of all bounded and linear maps from the Banach space $X$ to the Banach space $Z$. If $A$ is a subset of $X$,  the restriction of a mapping $f:X \to Z$ to $A$ is denoted by $f_{\mid_A}$. We say that $G:X\to Z$ is an extension of
$g:A\to Z$ if $G_{\mid_A}=g$. A mapping $f:A\rightarrow Z$ is $L$-Lipschitz whether $||f(x)-f(y)||\le L ||x-y||$ for every $x,y \in A$ and the Lipschitz constant of $f$ is $\Lip(f):=\sup\{\frac{||f(x)-f(y)||}{||x-y||}: \,x,y \in A, \, x\not=y \}$.

The $C^1$ smooth extension problem for real-valued functions defined on a subset of an infinite-dimensional Banach space
has been recently studied in  \cite{Azafrykeener}, where it  has been shown that, if $X$ is a Banach space  with separable dual,  $Y\subset X$  is a closed subspace of $X$ and
  $f: Y \to \Real$ is a $C^1$ smooth function, then there exists a ${C}^1$ smooth extension  of $f$ to $X$.
Also, a detailed review on the theory of smooth extensions is provided in  \cite{Azafrykeener}.
A generalization of the results in \cite{Azafrykeener} was given in \cite{MarLuis} for non-separable Banach spaces,  whenever the space satisfies a certain approximation property (property (*) for  $(X,\mathbb R)$; see definition below). When $X$ satisfies this approximation property, $A$ is a closed subset of $X$ and $f:A\to \Real$ is a function, the existence of a $C^1$ smooth extension of $f$ is characterized by the following property (called ``condition (E)''  in \cite{MarLuis}):

\begin{defn}
Let $X$ and $Z$ be  Banach spaces and $A\subset X$ be a closed subset.
\begin{enumerate}
\item We say that the mapping $f:A\to Z$ satisfies \textbf{the mean value condition}  if there exists a continuous map $D:A\to \LsXZ$ such that for every $y\in A$ and every $\varepsilon>0$, there is an open ball $B(y,r)$ in $X$ such that
\begin{equation*}
||f(z)-f(w)-D(y)(z-w)||\le \varepsilon||z-w||,
\end{equation*}
 for every $z,w\in A\cap B(y,r)$. In this case, we say that $f$ satisfies the mean value condition on $A$ for the map $D$.

\item We say that the mapping $f:A\to Z$ satisfies  \textbf{the mean value condition for a bounded map}   if it satisfies the mean value condition for a bounded and continuous map $D:A\to \LsXZ$, i.e.  $\sup\{||D(y)||: y\in A\}<\infty$.

 \end{enumerate}

\end{defn}
It is a straightforward consequence of the mean value theorem that, whenever  $f:A\to Z$  is the restriction of a $C^1$ smooth  mapping  $F:X\rightarrow Z$ ($C^1$ smooth and Lipschitz mapping), then
 $f:A\to Z$ satisfies the mean value condition  for ${F'}_{\mid_A}$ (the mean value condition for the bounded map ${F'}_{\mid_A}$, respectively).

In this note we adapt the proofs  given in the real-valued case \cite{Azafrykeener, MarLuis} to obtain, under certain conditions, $C^1$ smooth extensions of $C^1$ smooth mappings $f:A\rightarrow Z$. More precisely,  let us consider  the following properties.

\begin{defn}
\begin{enumerate}
\item The pair of Banach spaces $(X,Z)$ has  {\em \textbf{property (*)}} if there is a constant $C_0\ge 1$, which depends only on $X$ and $Z$, such that for every subset $A\subset X$,  every Lipschitz mapping $f:A\to Z$ and every $\varepsilon>0$,  there is a $C^1$ smooth and Lipschitz mapping $g:X\to Z$ such that  $||f(x)-g(x)||<\varepsilon$ for all $x\in A$ and $\Lip(g)\le C_0 \Lip(f)$.
\item The pair of Banach spaces $(X,Z)$ has  {\em \textbf{property (A)}} if there is a constant $C\ge 1$, which depends only on $X$ and $Z$, such that for every Lipschitz mapping $f:X\to Z$ and every $\varepsilon>0$, there exists a $C^1$ smooth and Lipschitz mapping $g:X\to Z$ such that $||f(x)-g(x)||<\varepsilon$ for all $x\in X$ and $\Lip(g)\le C \Lip(f)$.
\item The pair of Banach spaces $(X,Z)$ has  {\em \textbf{property (E)}} if there is a constant $K\ge 1$, which depends only on $X$ and $Z$, such that for every subset $A$ of $X$ and every Lipschitz mapping $f:A\to Z$, there exists a Lipschitz extension $F:X\to Z$ such that $\Lip(F)\le K\Lip(f)$.
\item A Banach space $X$ has   {\em \textbf{property (*)}}, {\em \textbf{property (A)}} or {\em \textbf{property (E)}}  whenever the pair $(X,\Real)$ does.
\end{enumerate}
\end{defn}

 \begin{rem}\label{remark:properties}
 \begin{enumerate} \item Clearly, a pair of Banach spaces (X,Z)  satisfies property (A) whenever it satisfies property (*). In Section 3 we shall prove that, in general, these properties are not equivalent.

 \item It is easy to prove that a pair of Banach spaces $(X,Z)$ satisfies property (*) provided that $(X,Z)$ satisfies properties (A) and (E). Moreover, if $Z$ is a dual Banach space, then $(X,Z)$ satisfies property (*) if and only if  $(X,Z)$ satisfies properties (A) and (E). Indeed, let us assume that $(X,Z)$ satisfies property (*) and consider a Lipschitz mapping $f:A\rightarrow Z$, where $A$ is a subset of $X$. Then, for every $n\in \mathbb N$, there is a $C^1$ smooth, Lipschitz mapping $f_n:X\rightarrow Z$ such that $||f(x)-f_n(x)||\le \frac{1}{n}$ for every $x\in A$ and $\Lip(f_n)\le C_0 \Lip(f)$.
     Then, for every $x\in X$, the sequence $\{f_n(x)\}_n$ is bounded. Since the closed balls in $(Z, ||\cdot||^*)$ are weak*-compact,
     there exists for every free ultrafilter $\mathcal U$ in $\mathbb N$,
     the weak*-limit $$\widehat{f}(x):=w^*-\lim_{\mathcal U} f_n(x).$$
     Clearly, $ \widehat{f}:X\rightarrow Z$ is an extension of $f:A\rightarrow Z$ and $\Lip(\widehat{f})\le C_0 \Lip(f)$.

 \item $X$ satisfies property (*) if and only if it satisfies property (A). Indeed, this is a consequence of the fact that $X$ always  has property (E):  if $A$ is a closed subset of $X$ and $f:A\to \Real$ is a Lipschitz function, then the function $F$ defined on $X$ as
\begin{equation*}
F(x)=\inf_{a\in A}\{f(a)+\Lip(f)||x-a||\}
\end{equation*}
is a Lipschitz extension of $f$ to  $X$ and $\Lip(F)=\Lip(f)$.
\end{enumerate}
\end{rem}

In Section 2, it is stated  that, if the pair of Banach spaces $(X,Z)$ satisfies  property (*), then every mapping $f:A\rightarrow Z$, where $A$  is a closed subset of $X$, is the restriction of a $C^1$ smooth mapping ($C^1$ smooth and Lipschitz mapping) $F:X\rightarrow Z$ if and only if $f$ satisfies the mean value condition (the mean value condition for a bounded map and $f$ is Lipschitz, respectively).

In Section 3 we give examples  of pairs of Banach spaces $(X,Z)$ satisfying property (*). In particular, when either (i) $X$ and $Z$ are Hilbert spaces with $X$ separable, or (ii) $X^*$ is  separable and $Z$ is a Banach space which is an absolute Lipschitz retract, or (iii)  $X=L_2$ and $Z=L_p$ with $1<p<2$, or (iv) $X=L_p$ and $Z=L_2$ with $2<p<\infty$. We also prove that property (*) is necessary and give an example of a pair of Banach spaces  satisfying   property (A) but not property (*).

In  Section 4, it is proved that every $C^1$ smooth mapping $f:Y\to Z$ defined on a closed subspace of $X$ admits a $C^1$ smooth extension to $X$,  whenever the pair of Banach spaces $(X,Z)$ satisfies property (*) and every bounded and linear operator $T:Y \rightarrow Z$  can be extended to a bounded and linear operator on $X$.  Moreover,  we obtain some  results on bounded and linear  extension morphisms on the Banach space $C_L^1(X,Z)$ of all $C^1$ smooth and Lipschitz mappings $f:X\rightarrow Z$.


\section{On smooth extension of mappings}

The main result of this note are the following theorems.

\begin{thm}\label{theorem:extension}
Let $(X,Z)$ be a pair of Banach spaces with  property (*), $A\subset X$  a closed subset of $X$ and  a mapping $f:A\to Z$. Then,
 $f$ satisfies the mean value condition
 if and only if there is a $C^1$ smooth
extension $G$ of $f$  to  $X$.
\end{thm}

\begin{thm}\label{theorem:Lipschitz:extension}
Let $(X,Z)$ be a pair of Banach spaces with  property (*), $A\subset X$  a closed subset of $X$ and  a mapping $f:A\to Z$. Then,
 $f$ is Lipschitz and satisfies the mean value condition for a bounded map  if and only if there is a $C^1$ smooth and Lipschitz extension
$G$ of $f$ to  $X$.

Moreover, if $f$ is Lipschitz and satisfies the mean value condition for a bounded map
$D:A\to \LsXZ$ with
$M:=\sup\{||D(y)||:y\in A\}<\infty$,
then we can obtain a $C^1$ smooth and Lipschitz extension $G$ with $\Lip(G)\le (1+C_0)(M+\Lip(f))$, where $C_0$ is the constant  given by  property (*) (which   depends only on
$X$ and $Z$).
\end{thm}

First of all, let us notice that if the pair $(X,Z)$ satisfies property (*), $X$ does too, i.e. there is a constant $C_0\ge 1$ (which  depends only on $X$) such that for every subset $A\subset X$, every  Lipschitz function $f:A\to\Real$ and every $\varepsilon>0$, there is a $C^1$ smooth and  Lipschitz function $g:X\to\Real$ such that $|g(x)-f(x)|<\varepsilon$ for all $x\in A$ and $\Lip(g)\le C_0\Lip(f)$. 
Indeed, let us take $e\in Z$ with $||e||=1$  and $\varphi\in Z^*$ with $||\varphi||=1$ and $\varphi(e)=1$. Let $f:A\to \Real$ be a $L$-Lipschitz function and $\varepsilon>0$. The mapping $h:A\to Z$ defined as $h(x)=f(x)e$ for all $x\in A$, is  $L$-Lipschitz.  Since the pair $(X,Z)$ satisfies property (*),  there exists a $C^1$ smooth and Lipschitz mapping $\widetilde{g} :X\to Z$ such that $||h(x)-\widetilde{g}(x)||<\varepsilon$ for all $x\in A$ and $\Lip(\widetilde{g})\le C_0L$.
The required  function $g:X\to \Real$ can be defined as $g(x):=\varphi(\widetilde{g}(x))$.

The proofs of the main theorems of this section follow the ideas of the  real-valued case. We shall need the following lemmas.


\begin{lem}\label{lemma:approx:bounded}
Let $(X,Z)$ be a pair of Banach spaces with property (*). Then,  for every subset $A\subset X$,  every Lipschitz mapping $f:A\to B_Z(0,R)$ (with $R\in (0,\infty)$) and every $\varepsilon>0$, there is a $C^1$ smooth and Lipschitz mapping $h:X\to Z$ such that
\begin{enumerate}
\item[(i)]  $||f(x)-h(x)||<\varepsilon$  {for every $x\in A$,}
\item[(ii)] $||h(x)|| < C_0\Lip(f)^{1/2}+R+\varepsilon$ for every $x\in X$, and
\item[(iii)] $\Lip(h)\le C_0 ((1+2C_0)\Lip(f)+2(R+\varepsilon)\Lip(f)^{1/2})$.
\end{enumerate}
\end{lem}
\begin{proof}
 Without loss of generality we may assume that  $\Lip(f)>0$. By property (*) there is a $C^1$ smooth and Lipschitz mapping $g:X\to Z$ such that
  \begin{equation*}
 ||f(x)-g(x)||<\varepsilon \qquad \text{for all $x\in A$, \ and} \quad \Lip(g)\le C_0 \Lip(f).
 \end{equation*}

 Let us define $W:=\{x\in X:\dist(x,\overline{A})\ge\frac{1}{\Lip(f)^{1/2}}\}$. Since $X$ satisfies property (*),  there is a $C^1$ smooth function $h_A:X\to [0,1]$ such that $h_A(x)=1$ whenever $x\in \overline{A}$, $h_A(x)=0$ whenever $x\in W$ and $\Lip(h_A)\le 2C_0 \Lip(f)^{1/2}$.    Let us define $h:X\to Z$ by $h(x):=g(x)h_A(x)$, which is $C^1$ smooth and $||f(x)-h(x)||<\varepsilon$ for all $x\in A$ (recall that $h_A(x)=1$ for  all $x\in A$).

Since $h_A(x)=0$ for all $x\in W$, we have that $h(x)=0$ for all $x\in W$. Also, $||h(x)||\le ||g(x)|| < R+ \varepsilon$ for all $x\in \overline{A}$. Now,  for each $x\not\in W$ there is $x_0\in \overline A$ such that $||x-x_0|| < \frac{1}{\Lip(f)^{1/2}}$ and thus,
\begin{equation*}
||g(x)||\le ||g(x)-g(x_0)||+||g(x_0)|| < C_0\Lip(f)||x-x_0||+ R+\varepsilon < C_0\Lip(f)^{1/2}+R+\varepsilon.
\end{equation*}
 Therefore, $||h(x)|| < C_0\Lip(f)^{1/2}+R+\varepsilon$ for every $x\in X$.  Now, if  $x\in \inte(W)$, then  $h'(x)=0$. Also,  if
 $x\not\in \inte(W)$, then
   \begin{align*}
 ||h'(x)||\le & ||g'(x)|| \hspace{0.03cm} ||h_A(x)||+||h_A'(x)|| \hspace{0.03cm} ||g(x)||\\
  \le & C_0 \Lip(f)+ 2C_0\Lip(f)^{1/2}(C_0\Lip(f)^{1/2}+R+\varepsilon) \\
 \le & C_0((1+2C_0)\Lip(f)+2(R+\varepsilon)\Lip(f)^{1/2}).
 \end{align*}
Thus, $\Lip(h)\le C_0((1+2C_0)\Lip(f)+2(R+\varepsilon)\Lip(f)^{1/2})$.
 \end{proof}

\begin{lem}\label{lemma:approx}
Let $(X,Z)$ be a pair of Banach spaces with  property (*). Then,
for every subset $A\subset X$, every continuous mapping  $F:X\to Z$ such that $F_{\mid_A}$ is Lipschitz, and every $\varepsilon>0$, there exists  a $C^1$ smooth  mapping $G:X\to Z$ such that
\begin{enumerate}
\item[(i)] $||F(x)-G(x)|| < \varepsilon$ for all $x \in X$,
\item[(ii)] $\Lip(G_{\mid_A})\le C_0\Lip(F_{\mid_A})$. Moreover,  $||G'(y)||\le C_0 \Lip(F_{\mid_A})$ for all $y\in A$, where $C_0$ is the constant given by property (*).
\item[(iii)]  In addition, if $F$ is Lipschitz, then  there exists a constant $C_1\ge C_0$ depending only on $X$ and $Z$, such that the mapping $G$ can be chosen to be Lipschitz on $X$ and $\Lip(G)\leq C_1 \Lip(F)$.
\end{enumerate}
\end{lem}
\noindent {\em Sketch of the proof.}
The proof is similar to the real-valued case (see \cite[Lemma 2.3]{MarLuis}).  Let us outline  the Lipschitz case. Let us apply property (*) to $F$ and  ${F}_{\mid_A}$ to obtain  $C^1$ smooth and Lipschitz mappings $g,\ h:X\to  Z$ such that
\begin{enumerate}
\item[(a)] $||{F}(x)-g(x)||<{\varepsilon}/{4}$ for all $x\in A$,
\item[(b)] $||F(x)-h(x)||<\varepsilon$ for all $x\in X$,
\item[(c)] $\Lip(g)\le C_0 \Lip(F_{\mid_A})$ and $\Lip(h)\le C_0 \Lip(F)$.
\end{enumerate}
There is a $C^1$ smooth and Lipschitz function $u:X \to[0,1]$ such that $u(x)=1$ whenever   $||F(x)-g(x)||\le {\varepsilon}/{4}$ and $u(x)=0$ whenever  $||F(x)-g(x)||\ge \varepsilon/2$, with $\Lip(u)\le \frac{ 9 C_0(\Lip(F)+C_0\Lip(F_{\mid_A}))}{\varepsilon}$ (see \cite{MarLuis} for details).
Then, the mapping $G:X\to Z$ defined as $G(x)=u(x)g(x)+(1-u(x))h(x)$ for every $x \in X$ is the required mapping with $C_1:=\frac{C_0}{2}(29+27C_0)$.
\qed


\begin{lem}\label{lema:clave}
Let $(X,Z)$ be a pair of Banach spaces with  property  (*),  a closed subset $A\subset X$   and
  a mapping $f:A\to B_Z(0,R)$ (with $R\in (0,\infty]$)  satisfying the mean value  condition for a map $D:A\to \LsXZ$.
 Then, for every $\varepsilon>0$ there exists a $C^1$ smooth  mapping $h:X\to B_Z(0,R+\varepsilon)$ such that
\begin{enumerate}
\item[{(i)}] $||f(y)-h(y)||< \varepsilon$ for all  $y\in A$,
\item[{(ii)}] $||D(y)-h'(y)||< \varepsilon$ for all $y\in A$, and
\item[(iii)] $\Lip(f-h_{\mid_A}) < \varepsilon$.
\end{enumerate}
\end{lem}

\begin{proof}
Since $A$ is closed, by a vector-valued version of the Tietze theorem (see for instance  \cite[Theorem 6.1]{Dugundji}) there is  a continuous extension
 $F:X\to B_Z(0,R)$  of $f$. Since $X$ is a Banach space, $A\subset X$ is a closed subset and $f$ satisfies {mean value condition} for $D:A\to \LsXZ$ on $A$, there exists $\{B(y_\gamma,r_\gamma)\}_{\gamma\in\Gamma}$  a covering of $A$ by open balls of $X$,    with centers $y_\gamma \in A$,  such that

\begin{equation}\label{condi}
||D(y)-D(y_\gamma)||\le\frac{\varepsilon}{8C_0}
\quad
\text{  and  }
\quad
||f(z)-f(w)-D(y_\gamma)(z-w)|| \le \frac{\varepsilon}{8C_0}  ||z-w|| ,\end{equation}
 for every $y,z,w\in B_\gamma\cap A$, where $B_{\gamma}:=B(y_\gamma,r_\gamma)$
 and $C_0$   is the constant given by  property (*).

Let us define $T_\gamma : X\rightarrow Z$  by
$T_\gamma(x)=f(y_\gamma)+D(y_\gamma)(x-y_\gamma)$, for every $x\in X$. Notice that  $T_\gamma$ satisfies the following properties:
    \begin{enumerate}
    \item[(B.1)] $T_\gamma$ is ${C}^\infty$ smooth on $X$,
    \item[(B.2)] $T'_\gamma(x)=D(y_\gamma)$ for all $x\in X$, and 
    \item[(B.3)]  $\Lip((T_\gamma-F){\mid_{B_{\gamma}\cap A}})\le\frac{\varepsilon}{8C_0}$, since  for all     $z,w \in B_\gamma \cap A$,
\begin{equation*}
||(T_\gamma-F)(z)-(T_\gamma-F)(w)||=||f(w)-f(z)-D(y_\gamma)(w-z)|| \le\frac{\varepsilon}{8C_0}  ||z-w||.
\end{equation*}
    \end{enumerate}

 Since $F:X\to B_Z(0,R)$ is a continuous mapping and $X$ admits $C^1$ smooth partitions of unity, there is a $C^1$ smooth mapping $F_0:X\to Z$ such that $||F(x)-F_0(x)||<\frac{\varepsilon}{2}$ for every $x\in X$.

Let us denote $B_0:=X\setminus A$, $\Sigma:=\Gamma\cup \{0\}$ (we assume $0\notin \Gamma$), and  $\mathcal{C}:=\{B_{\beta} : \beta\in \Sigma\}$, which is a covering of $X$.  By \cite{Rudin} and \cite[Lemma 2.2]{MarLuis}, there are an open refinement  $\{W_{n,\beta}\}_{n\in\mathbb{N},\beta\in\Sigma}$ of $\mathcal{C}=\{B_\beta : \beta\in\Sigma\}$ and a ${C}^1$ smooth and Lipschitz partition of unity $\{\psi_{n,\beta}\}_{n\in\mathbb{N},\beta\in\Sigma}$  satisfying: 

\begin{enumerate}
\item[{ (P1)}]
$\supp (\psi_{n,\beta}) \subset
W_{n,\beta}\subset B_\beta$; 

\item[{ (P2)}] $\Lip(\psi_{n,\beta})\le C_02^5(2^n-1)$ for every $(n,\beta)\in \mathbb N \times \Sigma$;
 and 
 
\item[{  (P3)}] for each $x\in X$ there is an open ball $B(x,s_x)$ of $X$ with center $x$ and radius $s_x>0$, and a natural number $n_x$ such that
     \begin{enumerate}
         \item[{ (1)}] if $i>n_x$, then $B(x,s_x)\cap W_{i,\beta}=\emptyset$ for every $\beta\in\Sigma$,
         \item[{ (2)}] if $i\leq n_x$, then $B(x,s_x)\cap W_{i,\beta}\neq\emptyset$ for at most one $\beta\in \Sigma$.
     \end{enumerate}
          \end{enumerate}


Let us define $L_{n,\beta}:=\max\{\Lip(\psi_{n,\beta}),1\}$ for every $n\in\mathbb{N}$ and $\beta\in\Sigma$.
Now, for every $n\in\mathbb{N}$ and  $\gamma\in\Gamma$, we apply Lemma \ref{lemma:approx} to $T_\gamma-F$ on $B_{\gamma}\cap A$  to obtain a $C^1$ smooth   mapping $\delta_{n,\gamma}:X\to Z$ so that
\begin{equation} \tag{C.1}
||T_\gamma(x)-F(x)-\delta_{n,\gamma}(x)||<\frac{\varepsilon}{2^{n+2}L_{n,\gamma}} \ \text{ for every } x\in X,
\end{equation}

\begin{equation} \tag{C.2}
 \|\delta'_{n,\gamma}(y)\|\le \frac{\varepsilon}{8} \ \text{ for every } y\in B_\gamma\cap A
\end{equation}

and
 \begin{equation}\tag{C.3}
 \Lip((\delta_{n,\gamma})_{\mid_{B_\gamma\cap A}})\le \frac{\varepsilon}{8}.
 \end{equation}

From inequality  \eqref{condi}, (B.2), (C.2) and (C.3),   we have, for all $y\in B_{\gamma}\cap A$,
\begin{equation*}
\|T'_\gamma(y)-D(y)-\delta'_{n,\gamma}(y)\| \leq \|T_\gamma'(y)-D(y)\|+\|\delta_{n,\gamma}'(y)\| \le  \frac{\varepsilon}{4},
\end{equation*}
and
\begin{equation*}
\Lip((T_\gamma-F-\delta_{n,\gamma})_{\mid_{B_\gamma\cap A}}) \le \frac{\varepsilon}{4}.
 \end{equation*}

\medskip

Let us define $\Delta_\beta^n:X\rightarrow Z$,
\begin{equation}\label{deltanbeta}
\Delta_\beta^n(x)=
\begin{cases}
F_0(x) & \text{ if } \beta=0, \\
T_\beta(x)-\delta_{n,\beta}(x)  & \text{ if } \beta\in \Gamma.
\end{cases}
\end{equation}
Thus,
$||F(x)-\Delta_\beta^n(x)||<\frac{\varepsilon}{2}$  whenever $n\in \mathbb{N}$,  $\beta\in\Sigma$ and  $x\in X$.
Now, we define
\begin{equation*}
h(x)=\sum_{(n,\beta)\in\mathbb{N} \times \Sigma} \psi_{n,\beta}(x)\Delta^n_\beta(x).
\end{equation*}
Since $\{\psi_{n,\beta}\}_{n\in\mathbb{N},\beta\in\Sigma}$ is locally finitely nonzero,  $h$ is ${C}^1$ smooth.
Now, if $x\in X$, then
\begin{equation*}
||F(x)-h(x)||\leq \sum_{(n,\beta)\in\mathbb{N} \times \Sigma}\psi_{n,\beta}(x)||F(x)-\Delta_\beta^n(x)||\leq \sum_{(n,\beta)\in\mathbb{N} \times \Sigma} \psi_{n,\beta}(x)\frac{\varepsilon}{2}<\varepsilon.
\end{equation*}
Therefore, $||h(x)|| < R+\varepsilon$ for all $x\in X$ (recall that $||F(x)||\le R$ for all $x\in X$).
Following the proof of \cite[Theorem 3.3]{MarLuis}, it can be cheked that
\begin{equation*}
 \|D(y)-h'(y)\| <\varepsilon \text{ for all } y\in A  \text{ and }  \Lip(f-h_{\mid_A})<\varepsilon.
 \end{equation*}
\end{proof}


\begin{lem}\label{theorem:clave}
Let $(X,Z)$ be a pair of Banach spaces with  property (*), a closed subset  $A\subset X$     and  a Lipschitz mapping $f:A\to B_Z(0,R)$ (with $R\in (0,\infty]$)  satisfying the mean value  condition for a bounded map $D:A\to \LsXZ$ with $M=\sup\{||D(y)||:y\in A\}<\infty$.  Then, for every $\varepsilon>0$ there is a  $C^1$ smooth and Lipschitz mapping $g:X\to Z$ such that
\begin{enumerate}
\item[(i)] $||f(y)-g(y)||<\varepsilon$ { for every}  $y\in A$,
\item[(ii)] $||D(y)-g'(y)||<\varepsilon$ for every  $y\in A$,
\item[(iii)] $\Lip(f-g_{\mid_A})<\varepsilon$,
\item[(iv)] $||g(x)||  < C_0\Lip(f)^{1/2}+ R+\varepsilon$ for every $x\in X$,
\item[(v)]  $\Lip(g) \le C_0( (1+2C_0)\Lip(f)+2(R+\varepsilon)\Lip(f)^{1/2}+M )+\varepsilon$ whenever $R<\infty$, and $\Lip(g)\le C_0(M+\Lip(f))+\varepsilon$ whenever $R=+\infty$;  where $C_0$ is the constant given by property (*).
\end{enumerate}
\end{lem}

\begin{proof}
Let us suppose that $R<+\infty$, and take $0<3\varepsilon'<\varepsilon$. By Lemma \ref{lema:clave} there is a $C^1$ smooth mapping $h:X\to B_Z(0,R+\varepsilon')$ such that
\begin{enumerate}
\item[{(i)}] $||f(y)-h(y)||< \varepsilon'$ for all $y\in A$,
\item[{(ii)}] $||D(y)-h'(y)||< \varepsilon'$ for all $y\in A$, and
\item[{(iii)}] $\Lip(f-h_{\mid_A}) < \min\{\frac{\varepsilon'}{C_0(1+2C_0)},(\frac{\varepsilon'}{2C_0(R+2\varepsilon')})^{2}\}$.
\end{enumerate}


Since $h$ is $C^1$ smooth on $X$,   there exists $\{B(y_\gamma,r_\gamma)\}_{\gamma\in\Gamma}$  a covering of $A$ by open balls of $X$,    with centers $y_\gamma \in A$ such that

\begin{equation}\label{condii}
||h(y)-h(y_\gamma)||\le\frac{\varepsilon'}{8C_0}
\quad
\text{  and  }
\quad 
||h'(y)-h'(y_\gamma)||\le\frac{\varepsilon'}{8C_0},  \text{  for every   } y\in B_\gamma,
\end{equation}
where $B_{\gamma}:=B(y_\gamma,r_\gamma)$ and $ C_0$ is the constant given by property (*) (which depends only on $X$ and $Z$). Let us define $T_\gamma$  by
$T_\gamma(x)=h(y_\gamma)+h'(y_\gamma)(x-y_\gamma)$, for $x\in X$. Notice that  $T_\gamma$ satisfies the following properties:
    \begin{enumerate}
    \item[(B.1)] $T_\gamma$ is ${C}^\infty$ smooth on $X$,
    \item[(B.2)] $T'_\gamma(x)=h'(y_\gamma)$ for all $x\in X$,  
    \item[(B.3)]  $\Lip((T_\gamma-h){\mid_{B_{\gamma}}})\le\frac{\varepsilon'}{8C_0}$, and
    \item[(B.4)] $||T_\gamma'(x)||=||h'(y_\gamma)|| \le ||D(y_\gamma)||+\varepsilon' \le    M+\varepsilon'$ for every $x\in X$.
    \end{enumerate}

Let us define $B_0:=X\setminus A$, $\Sigma:=\Gamma\cup \{0\}$ (we assume $0\notin \Gamma$), and  $\mathcal{C}:=\{B_{\beta} : \beta\in \Sigma\}$, which is an open covering of $X$. Following  the proof of Lemma \ref{lema:clave},  we obtain an open refinement  $\{W_{n,\beta}\}_{n\in\mathbb{N},\beta\in\Sigma}$ of $\mathcal{C}=\{B_\beta : \beta\in\Sigma\}$ and a ${C}^1$ smooth and Lipschitz partition of unity $\{\psi_{n,\beta}\}_{n\in\mathbb{N},\beta\in\Sigma}$  satisfying conditions (P1), (P2) and (P3).

Let us define $L_{n,\beta}:=\max\{\Lip(\psi_{n,\beta}),1\}$ for every $n\in\mathbb{N}$ and $\beta\in\Sigma$.
Now, for every $n\in\mathbb{N}$ and  $\gamma\in\Gamma$, we apply property (*) to $T_\gamma-h$ on $B_{\gamma}$ in order to obtain a $C^1$ smooth   mapping $\delta_{n,\gamma}:X \to Z$ so that
\begin{equation} \tag{C.1}
||T_\gamma(x)-h(x)-\delta_{n,\gamma}(x)||<\frac{\varepsilon'}{2^{n+2}L_{n,\gamma}} \qquad \text{for every } x\in B_\gamma
\quad \text{ and }
 \end{equation}
\begin{equation} \tag{C.2}
\Lip(\delta_{n,\gamma})\le C_0
\Lip((T_\gamma-h)_{\mid_{B_\gamma}})\le \frac{\varepsilon'}{8}.
\end{equation}
In particular,
\begin{equation}
||T_\gamma(x)-\delta_{n,\gamma}(x)||<||h(x)||+\frac{\varepsilon'}{2^{n+2}L_{n,\gamma}}<R+2\varepsilon' \qquad \text{for every } x\in B_\gamma.
\end{equation}
%
%
From inequality  \eqref{condii}, (B.2) and (C.2)  and  for every  $y\in B_{\gamma}$, we have
\begin{equation*}
\|T'_\gamma(y)-h'(y)-\delta'_{n,\gamma}(y)\| \leq \|T_\gamma'(y)-h'(y)\|+\|\delta_{n,\gamma}'(y)\| \le \frac{\varepsilon'}{4}.
\end{equation*}
Therefore,
\begin{equation*}
\Lip((T_\gamma-h-\delta_{n,\gamma})_{\mid_{B_\gamma }}) \le \frac{\varepsilon'}{4}.
 \end{equation*}

Now, since $\Lip(h_{\mid_A})\le \Lip(f)+(\varepsilon'/C_0)^2$, let us apply Lemma \ref{lemma:approx:bounded}  to $h_{\mid_A}:A\to B_Z(0,R+\varepsilon')$ to obtain $C^1$ smooth and Lipschitz mappings $F^n_0:X\to B_Z(0,C_0\Lip(f)^{1/2}+ R+3\varepsilon')$ such that $||F^n_0(z)-h(z)|| <\frac{\varepsilon'}{2^{n+2}L_{n,0}}$  for all $z\in A$ and $n\in \Natural$, and
\begin{equation*}
 \Lip(F^n_0)  \le C_0((1+2C_0)\Lip(h_{\mid_A})+2(R+2\varepsilon')\Lip(h_{\mid_A})^{1/2}).
\end{equation*}
From condition (iii) above, we deduce
\begin{equation*}
\Lip(F^n_0)  \le C_0((1+2C_0)\Lip(f)+2(R+2\varepsilon')\Lip(f)^{1/2})+2\varepsilon'.
\end{equation*}

Let us define $\Delta_\beta^n: X\rightarrow Z$ and $g: X\rightarrow Z$ as
\begin{equation}\label{deltanbeta}
\Delta_\beta^n(x)=
\begin{cases}
F^n_0(x) & \text{ if } \beta=0, \\
T_\beta(x)-\delta_{n,\beta}(x)  & \text{ if } \beta\in \Gamma,
\end{cases}
\qquad \text{and} \qquad g(x)=\sum_{(n,\beta)\in\mathbb{N} \times \Sigma} \psi_{n,\beta}(x)\Delta^n_\beta(x).
\end{equation}
Since $\{\psi_{n,\beta}\}_{n\in\mathbb{N},\beta\in\Sigma}$ is locally finitely nonzero, the mapping $g$ is ${C}^1$ smooth. It is clear that
\begin{equation*}
||g(x)||\le \sum_{(n,\beta)\in\mathbb{N} \times \Sigma} \psi_{n,\beta}(x)||\Delta^n_\beta(x)||Ê< C_0\Lip(f)^{1/2}+ R+\varepsilon \qquad \text{for all $x\in X$.}
\end{equation*}
The proofs of $||h(y)-g(y)||<\varepsilon'$, \,  $||h'(y) -g'(y)|| <\varepsilon'$  for all $y\in A$ and  $\Lip((h-g)_{\mid_A})<\varepsilon'$ follow along the same lines as \cite[Theorem 3.3]{MarLuis}. Thus, $||f(y)-g(y)||<\varepsilon$ for all $y\in A$, $||D(y) -g'(y)|| <\varepsilon$  for $y\in A$, and  $\Lip(f-g_{\mid_A})<\varepsilon$.

In addition, since $||T_\gamma'(x)||+||\delta_{n,\gamma}'(x)||\le  M+9\varepsilon'/8$,
\begin{align*}
||(\Delta_\beta^n)'(x)||\leq & \max\{C_0 ((1+2C_0)\Lip(f)+2(R+2\varepsilon')\Lip(f)^{1/2}) +2\varepsilon',  M+9\varepsilon'/8\} \\
 \le & C_0 ((1+2C_0)\Lip(f)+2(R+2\varepsilon')\Lip(f)^{1/2}+M)+2\varepsilon'.
 \end{align*}
Let us check that $g$ is Lipschitz. From   the fact that  $\sum_{(n,\beta)\in F_x}\psi'_{n,\beta}(x)=0$ for all $x\in X$, where $F_x:=\{(n,\beta)\in \mathbb N\times \Sigma: x\in \supp  (\psi_{n,\beta})\}$ and the fact (P3) we deduce that
\begin{align*}\label{cotaparaG'}
& ||g'(x)|| \le  \sum_{(n,\beta)\in F_x}||\psi'_{n,\beta}(x)||\,||h(x)-\Delta_\beta^n(x)||  + \sum_{(n,\beta)\in F_x}\psi_{n,\beta}(x)||(\Delta_\beta^n)'(x)|| \\  \notag
  &  \le  \sum_{\{n:(n,\beta(n))\in F_x\}}L_{n, \beta(n)}\, \frac{\varepsilon'}{2^{n+2}L_{n, \beta(n)}} \\ \notag
&      + \sum_{\{n:(n,\beta(n))\in F_x\}}\psi_{n,\beta(n)}(x)(C_0((1+2C_0)\Lip(f)+2(R+2\varepsilon')\Lip(f)^{1/2}+M)+2\varepsilon')  \\  \notag
& <C_0 ((1+2C_0)\Lip(f)+2(R+2\varepsilon')\Lip(f)^{1/2}+M) +3\varepsilon',
\end{align*}
for all $x\in X$, where $\beta(n)$ is the only index $\beta$ (if there exists) satisfying condition (P3)-(2) for $x$. Thus, $\Lip(g)\le C_0 ((1+2C_0)\Lip(f)+2(R+\varepsilon)\Lip(f)^{1/2}+M)+\varepsilon$. (Recall, that here we do not assume $\varepsilon < \Lip(f)$.)


If   $R=+\infty$, we apply property (*) to $h_{\mid_A}:A\to Z$ in order to obtain $C^1$ smooth mappings $F^n_0:X\to Z$ such that $||h(x)-F^n_0(x)||<\frac{\varepsilon'}{2^{n+2}L_{0,\beta}} $ on $A$, and   $\Lip(F^n_0)\le C_0\Lip(h_{\mid_A})\le C_0\Lip(f)+\varepsilon'$. Thus,  $||(\Delta_\beta^n)'(x)||\leq \max\{C_0\Lip(f)+\varepsilon',  M+9\varepsilon'/8\}\le C_0 (\Lip(f)+M) +9\varepsilon'/8$ and $||g'(x)||\le C_0(\Lip(f)+M)+\varepsilon$ for every $x\in X$.
\end{proof}


\medskip

\noindent {\em Proofs of Theorems \ref{theorem:extension} and \ref{theorem:Lipschitz:extension}.} Let us assume that  the mapping $f:A\to Z$ satisfies the mean value condition and consider $0<\varepsilon<1$. 
Then, by Lemma \ref{lema:clave} there exists a $C^1$ smooth mapping $G_1:X\to Z$ such that  if  $g_1:={G_1}_{\mid_A}$, then
\begin{enumerate}
\item[(i)]  $||f(y)-g_1(y)|| < \frac{\varepsilon}{2^4C_0}$ for every $y\in A$,
\item[(ii)] $||D(y)-G_1'(y)|| < \frac{\varepsilon}{2^4C_0}$ for every $y\in A$, and
\item[(iii)] $\Lip(f-g_1) < \min\{\frac{\varepsilon}{2^4C_0(1+2C_0)}, (\frac{\varepsilon}{2^4C_0})^2\}$.
\end{enumerate}

Notice that the mapping $f-{g_1}$   satisfies the {mean value condition for the bounded map}
  $D-G_1':A\to \LsXZ$ with $\sup\{||D(y)-G_1'(y)||:y\in A\}\le \frac{\varepsilon}{2^4C_0}$. Let us apply 
  Lemma  \ref{theorem:clave} to $f-g_1$ to obtain  a $C^1$ smooth mapping $G_2:X\to Z$ such that if
 $g_2:={G_2}_{\mid_A}$, then
   \begin{enumerate}
\item[(i)]  $||{(f-g_1)}(y)-g_2(y)|| < \frac{\varepsilon}{2^5C_0}$ for every $y\in A$,
\item[(ii)]$||D(y)-(G_1'(y)+G_2'(y))|| < \frac{\varepsilon}{2^5C_0}$ for every $y\in A$,
\item[(iii)] $\Lip(f-(g_1+g_2)) <  \min\{\frac{\varepsilon}{2^5C_0(1+2C_0)}, (\frac{\varepsilon}{2^5C_0})^2\}$, 
\item[(iv)] $||G_2(x)||\le C_0 \frac{\varepsilon}{2^4C_0}+\frac{\varepsilon}{2^4C_0}+\frac{\varepsilon}{2^5 C_0} \le \frac{\varepsilon}{2^2}$ for all $x\in X$, and
\item[(v)]${ \Lip(G_2)  \le C_0((1+2C_0)\frac{\varepsilon}{2^4C_0(1+2C_0)}}+2(\frac{\varepsilon}{2^4C_0}+\frac{\varepsilon}{2^5C_0})\frac{\varepsilon}{2^4C_0}+\frac{\varepsilon}{2^4C_0})  +\frac{\varepsilon}{2^5C_0} \le \frac{\varepsilon}{2^2}$.
\end{enumerate}


 By induction, we find a sequence  $G_n:X\to Z$  of $C^1$ smooth mappings satisfying for  $n\ge 2$, where $g_n:={G_n}_{\mid_A}$, 
\begin{itemize}
\item[(i)]  $||{(f-\sum_{i=1}^{n-1} g_i)}(y)-g_n(y)|| < \frac{\varepsilon}{2^{n+3}C_0}$ for every $y\in A$,
\item[(ii)] $||D(y)-\sum_{i=1}^n G_i'(y)|| < \frac{\varepsilon}{2^{n+3}C_0}$ for every $y\in A$,
\item[(iii)] $\Lip (f-\sum_{i=1}^{n} g_i) < \min\{\frac{\varepsilon}{2^{n+3}C_0(1+2C_0)}, (\frac{\varepsilon}{2^{n+3}C_0})^2\}$, 
\item[(iv)] $||G_n(x)||\le \varepsilon/2^n$ for all $x\in X$, and
\item[(v)] $\Lip(G_n)\le  \varepsilon/2^{n}$.
\end{itemize}

It can be cheked as in \cite{Azafrykeener} and \cite{MarLuis} that the mapping $G:X\to Z$ defined  as $G(x):=\sum_{n=1}^\infty G_n(x)$ is $C^1$ smooth and it is an extension of $f$ to $X$.

\smallskip

Let us now consider $f:A\to Z$  a Lipschitz mapping satisfying the mean value condition for a bounded map $D:A\to \LsXZ$
  with $M:=\sup\{||D(y)||:y\in A\}<\infty$. We can assume  that $\varepsilon \le \frac{16(M+\Lip(f))}{9}$ (if $\Lip(f)=0$, the extension is obvious).  By Lemma \ref{theorem:clave}, there exists a $C^1$ smooth mapping $G_1:X\to Z$ such that if  $g_1:={G_1}_{\mid_A}$
\begin{enumerate}
\item[(i)]  $||f(y)-g_1(y)|| < \frac{\varepsilon}{2^4C_0}$ for every $y\in A$,
\item[(ii)] $||D(y)-G_1'(y)|| < \frac{\varepsilon}{2^4C_0}$ for every $y\in A$,
\item[(iii)] $\Lip(f-g_1) < \min\{\frac{\varepsilon}{2^4C_0(1+2C_0)}, (\frac{\varepsilon}{2^4C_0})^2\}$, and
\item[(iv)] $\Lip(G_1)\le C_0(M+\Lip(f)) + \frac{\varepsilon}{2^4}$.
\end{enumerate}

The mappings $G_n:X\to Z$ for $n\ge 2$ are defined as in the general case.  It can be checked  that the mapping $G:X\to Z$ defined as $G(x):=\sum_{n=1}^\infty G_n(x)$ is $C^1$ smooth, is an extension of $f$ to $X$ and
\begin{equation*}
\Lip(G)\le  C_0(M +\Lip(f)) + \frac{\varepsilon}{2^4}+ \sum_{n=2}^\infty \frac{\varepsilon}{2^{n}}   \le (1+C_0)(M+\Lip(f)).
\end{equation*}
\qed


\medskip

Let us now give an application  to $C^1$ Banach manifolds.
\begin{defn}
Let us consider  $C^1$  Banach manifolds  $M$ and $N$  modeled on the Banach spaces $X$ and $Z$ respectively, a subset $A$ of $ M$
and a  continuous map $f:M\to N$.
We say that the mapping $f:A\to N$ satisfies {\em {\bf the mean value condition}} on $A$
if  for every $x\in A$, there are (equivalently, for all) $C^1$ smooth
 charts $\varphi:U\to X$ and $\psi:V\to Z$ with $U$ an open subset of $M$,
  $x\in U$, $V$ an open subset of $N$, $f(U)\subset V$,  such that
 the mapping $\psi\circ f\circ \varphi^{-1}$ satisfies the mean value condition on $\varphi(A)\cap \varphi(U)$.

 \end{defn}

The proof of the following Corollary is similar to the real-valued case \cite{Azafrykeener, MarLuis}. Recall that a paracompact ${C}^1$ manifold $M$ modeled on a Banach space $X$ admits  ${C}^1$ smooth partitions of unity whenever the Banach space $X$  does.

\begin{cor}
Let $M$ be a paracompact $C^1$ Banach manifold modeled  on the Banach space $X$ and $N$ a $C^1$ Banach manifold modeled on the Banach space $Z$. Assume that the pair $(X,Z)$ satisfies property (*).  Let $A$ be a closed subset of $M$ and $f:A\to N$ a mapping. Then,
 $f$ satisfies the  mean value condition on $A$
 if and only if there is a $C^1$ smooth
extension $G:M\to N$ of $f$.
\end{cor}

\section{On the properties (*), (A) and (E)}


In this section, we present examples of pairs of Banach spaces $(X,Z)$ satisfying property (*). The first examples are pairs of Banach spaces satisfying  properties (A) and (E) and thus  property (*).

\begin{ex}
Let $X$ and $Z$ be  Banach spaces such that $X$ is finite dimensional.  Then, the pair  $(X,Z)$ satisfies properties (A) and (E).
\noindent {\em On the one hand, W.B. Johnson, J. Lindenstrauss and G. Schechtman   have shown in  \cite{JohLindSche}  that every pair of Banach space $(X,Z)$ with $X$ $n$-dimensional satisfies property (E) with  constant $K(n)\ge 1$ (which depends only on the dimension of $X$). On the other hand, the classical convolution techniques for smooth approximation in finite dimensional spaces provide property (A) for $(X,Z)$.}
\end{ex}

\begin{ex}
Let $X$ and $Z$ be  Hilbert spaces with $X$ separable. Then $(X,Z)$ satisfies the properties  (A) and (E).
{\em M.D. Kirszbraun has shown in \cite{Kirs} (see \cite[Theorem 1.12]{BenyLind}) that the pair $(X,Z)$ satisfies property (E) with $K=1$, whenever $X$ and $Z$ are Hilbert spaces.  Also, R. Fry has proven in \cite{Fry}  (see also  \cite[Theorem H]{HJ}) that $(X,Z)$  satisfies property (A) when $X$ is a separable Hilbert space. }
\end{ex}


\begin{ex}
The pairs $(L_2,L_p)$ for  $1<p<2$ and   $(L_p,L_2)$ for $2<p<\infty$ satisfy properties
(A) and (E). {\em  K. Ball showed  that for every $1<p<2$ the pair  $(L_2,L_p)$ satisfies property (E) with
constant $K(p)\ge 1$ depending only on $p$ \cite{Ball}.  I. G. Tsar'kov  proved that for every $2<p<\infty$ the pair
$(L_p,L_2)$ satisfies property (E) with   constant $K(p)\ge 1$ depending only on $p$  \cite{Tsar}.
Also,  the results in  \cite[Theorem 1]{Fry} yield the fact that $(X,Z)$ satisfies property (A).}
\end{ex}


Recall that a subset $A$ of a metric space $Z$ is called a \emph{Lipschitz retract} of $Z$ if there is a \emph{Lipschitz retraction} from $Z$ to $A$, i.e. there is a Lipschitz map $r:Z\to A$ such that $r_{\mid_A}=id_A$. A metric space $Z$ is called an \emph{absolute Lipschitz retract} if it is a Lipschitz retract of any metric space $W$ containing $Z$. The spaces $(c_0(\Natural),||\cdot||_\infty)$, $(\ell_\infty(\Natural), ||\cdot||_\infty)$ and  $(C(K),||\cdot||_\infty)$  for every compact metric space $K$ are  absolute Lipschitz retracts (see \cite{BenyLind} for more information and examples of absolute Lipschitz retracts). An  absolute Lipschitz retract space satisfies the following Lipschitz extension property.

\begin{prop}\label{proposition:retract}\cite[Proposition 1.2]{BenyLind}
Let $Z$ be a metric space. The following are equivalent:
\begin{enumerate}
\item[(i)] $Z$ is an absolute Lipschitz retract.
\item[(ii)] There is $K\ge 1$, which only  depends on $Z$, such that for every metric space $X$, every subset $A\subset X$ and every Lipschitz mapping $f:A\to Z$, there is a Lipschitz extension $F:X\to Z$ of $f$ such that $\Lip(F)\le K \Lip(f)$.
\end{enumerate}
\end{prop}

We  obtain the following proposition (see \cite{HJ}).
\begin{prop}\label{prop:approx}
Let $X$  be a Banach space such that there are a set $\Gamma\neq \emptyset$  and a bi-Lipschitz homeomorphism $\varphi:X\to c_0(\Gamma)$ with $C^1$ smooth  coordinate functions $e^*_\gamma\circ \varphi:X\to \Real$. Let  $Z$ be a Banach space which is an absolute Lipschitz retract. Then the pair $(X,Z)$ satisfies  properties (A) and (E).
\end{prop}
\begin{proof}
Let us take the mapping $f\circ \varphi^{-1}:\varphi(X)\to Z$ which is $\Lip(\varphi^{-1})\Lip(f)$-Lipschitz. By Proposition \ref{proposition:retract}, there is a Lipschitz extension $\widetilde{f}:c_0(\Gamma)\to Z$ of $f\circ \varphi^{-1}$ with $\Lip(\widetilde{f})\le K \Lip(\varphi^{-1})\Lip(f)$ and $K$ is the constant given in Proposition \ref{proposition:retract}. Now, from the results in  \cite{HJc0} we can find a  $C^\infty$ smooth and Lipschitz mapping $h:c_0(\Gamma)\to Z$  which locally depends on finitely many coordinate functionals $\{e_\gamma^*\}_{\gamma \in \Gamma}$, such that $||\widetilde{f}(x)-h(x)||<\varepsilon$ and $\Lip(h)=\Lip(\widetilde{f})$.
Let us define  $g:X\to Z$ as $g(x):=h(\varphi(x))$ for every  $x\in X$. The mapping $g$
 is $C^1$ smooth, $||f(x)-g(x)||<\varepsilon $ for all $x\in X$ and $\Lip(g)\le C\Lip(f)$, with $C:= K \Lip(\varphi) \Lip(\varphi^{-1})$.
 \end{proof}

This provides the following example.

\begin{ex}
Let $X$ and $Z$ be Banach spaces such that $X^*$ is separable and $Z$ is an absolute Lipschitz retract. Then,  the pair $(X,Z)$  satisfies  properties (A) and (E).  {\em Notice that P.  H\'ajek and M. Johanis \cite{HJ} proved the existence of a bi-Lipschitz homeomorphism with $C^k$ smooth coordinate functions  in every separable Banach space with a $C^k$ smooth and Lipschitz bump function. }
\end{ex}


Now, with these examples, Theorem \ref{theorem:extension} and Theorem \ref{theorem:Lipschitz:extension}, we obtain
the following consequence.
\begin{cor} \label{extension:c_0}
Let $X$ and $Z$ be Banach spaces and assume that at least one of the following conditions holds:
\begin{itemize}
\item[(i)] $X$ is finite dimensional,
\item[(ii)] $X$ and $Z$ are Hilbert spaces and $X$ is separable,
\item[(iii)] $X=L_2$ and $Z=L_p$ with $1<p<2$,
\item[(iv)] $X=L_p$ and $Z=L_2$ with $2<p<\infty$,
\item[(v)]  there are a set $\Gamma\neq \emptyset$  and a bi-Lipschitz homeomorphism $\varphi:X\to c_0(\Gamma)$ with $C^1$ smooth coordinate functions (for example, when  $X^*$ is separable), and $Z$ is an absolute Lipschitz retract.
\end{itemize}
Let $A$ be a closed subset of $X$ and $f:A\to Z$ a mapping.
Then,
 $f$ satisfies the mean value condition (mean value condition for a bounded map and
  $f$ is Lipschitz) on $A$
 if and only if there is a $C^1$ smooth ($C^1$ smooth and Lipschitz, respectively)
extension $G$ of $f$ to  $X$.

Moreover, if $f$ is Lipschitz and  satisfies the mean value condition for a bounded map  $D:A\to \LsXZ$ with
$M:=\sup\{||D(y)||:y\in A\}<\infty$,
then we can obtain a $C^1$ smooth and Lipschitz extension $G$ with $\Lip(G)\le (1+C_0)(M+\Lip(f))$, where $C_0\ge 1$  is the constant given by property (*) (which depends only on $X$ and $Z$).
\end{cor}


Let us now prove that  property (*) is  necessary  to obtain $C^1$ smooth and Lipschitz extensions. We also obtain  an example of a pair of Banach spaces satisfying property (A) but  not property (*). Thus, it does not admit $C^1$ smooth and Lipschitz extension.

\begin{prop}
Let  $(X,Z)$ be a pair of Banach spaces such that there is a constant $C\ge 1$, which only depends on $X$ and $Z$,  such that for every closed subset $A\subset X$  and every Lipschitz mapping $f:A\rightarrow Z$  satisfying the mean value condition for a bounded map $D$ with $M=\sup\{||D(y)||:y\in A\}<\infty$, there exists a $C^1$ smooth and Lipschitz extension $G$ of $f$ to $X$ with  $\Lip(G)\le C(M+\Lip(f))$.  Then, the pair $(X,Z)$ satisfies property (*). 
Therefore, by Theorem \ref{theorem:Lipschitz:extension}, the above assumption  is equivalent to property (*).
\end{prop}

\begin{proof}
Let $A$ be a subset of $X$,  $f:A\to Z$ a $L$-Lipschitz mapping and $\varepsilon>0$. Let us take  a  $\frac{\varepsilon}{(C+1)L}$-net in $A$ which we shall denote by $N$, i.e.  a subset $N$ of $A$ such that  (i)
 $||z-y||\ge \frac{\varepsilon}{(C+1)L}$ for every $z,y\in N$, (ii)   for every $x\in A$ there is a point $y\in N$ such that $||x-y||\le \frac{\varepsilon}{(C+1)L}$. Clearly,  $N$ is a closed subset of $X$ and $f_{\mid_N}:N\to Z$ is a $L$-Lipschitz mapping on $N$  satisfying the mean value condition for the bounded map given by $D(x)=0 \in \LsXZ$ for every $x\in N$. Then, by assumption, there exists a $C^1$ smooth and $CL$-Lipschitz  mapping $G:X\to Z$ such that $G_{\mid_N}=f_{\mid_N}$. For  any $x\in A$, let us choose $y\in N$ such that $||x-y||\le \frac{\varepsilon}{(C+1)L}$. Then, $G(y)=f(y)$ and
\begin{equation*}
||f(x)-G(x)||\le||f(x)-f(y)|| +  ||G(x)-G(y)|| \le (L+CL)||x-y||\le \varepsilon.
\end{equation*}
\end{proof}


\begin{ex}
Although the pair $(L_p,L_2)$ with $1<p<2$  satisfies  property (A) (see \cite{Fry}),
it does not satisfy property (E) \cite{Naor}.  {\em Thus, Remark \ref{remark:properties}\hspace{0.6 mm}(2) implies that the pair  $(L_p,L_2)$ does not satisfy property (*), and the above proposition reveals that there exist Lipschitz mappings $h:A\rightarrow L_2$ defined on closed subsets  $A$ of $L_p$ satisfying  the mean value condition
on $A$ for a bounded map which cannot be extended to  $C^1$ smooth and Lipschitz mappings  on $L_p$. In particular,  property (A) is a necessary condition but it is not a sufficient condition to obtain $C^1$ smooth and Lipschitz extensions, and properties (*) and (A) are not equivalent in general.}

\end{ex}


\section{Smooth extension from subspaces}

Finally, let us make a brief comment on the extension of $C^1$ smooth mappings defined on a subspace. Let $X$ and $Z$ be Banach spaces and $Y$ a closed subspace of $X$. If every $C^1$ smooth  mapping $f:Y\to Z$ can be extended to a $C^1$ smooth  mapping $F:X\to Z$, then for every bounded and linear operator $T:Y\to Z$ there is a bounded and linear operator $\widetilde{T}:X\to Z$ such that $\widetilde{T}_{\mid_Y}=T$. Moreover,
assume that every  $C^1$ smooth and Lipschitz mapping $f:Y\to Z$ can be extended to a $C^1$ smooth and Lipschitz mapping $F:X\to Z$ with $\Lip(F)\le C \Lip(f)$ with $C$ depending only on $X$ and $Z$. Then, for every bounded and linear  operator $T:Y\to Z$
there is a bounded and linear operator $\widetilde{T}:X\to Z$ such that $\widetilde{T}_{\mid_Y}=T$ and $||\widetilde{T}||_{\LsXZ}\le C ||T||_{\LsYZ}$.
Indeed, it is enough to consider $\widetilde{T}=G'(0)$, where $G$ is the extension mapping of $T$ given by the assumptions.


\begin{defn}
We say that the pair of Banach spaces $(X,Z)$ satisfies the \textbf{linear extension property} if  there is $\lambda\ge 1$, which depends only on $X$ and $Z$, such that  for every closed subspace $Y\subset X$ and every bounded and linear operator $T:Y\to Z$, there is a bounded and linear operator $\widetilde{T}:X\to Z$ such that $\widetilde{T}_{\mid_Y}=T$ and $||\widetilde{T}||_ {\LsXZ}\le \lambda  ||T||_{\LsYZ}$.
\end{defn}

\begin{ex}
\begin{enumerate}
\item[(i)] Maurey's extension theorem \cite{Maurey} asserts that the pair of Banach spaces $(X,Z)$ satisfies the linear extension property whenever $X$ has type $2$ and $Z$ has cotype $2$. Therefore, $(L_2,L_p)$ for $1<p<2$   and  $(L_p,L_2)$ for $2<p<\infty$  satisfy the linear extension property (recall  that  $L_p$ has type $2$  for $2\le p <\infty$ and cotype $2$ for $1<p \le2$, see \cite{AlbiacKalton}).

\item[(ii)] For every compact metric space $K$, every non-empty set $\Gamma$ and
 $1<p<\infty$, the pairs $(c_0(\Gamma),C(K))$ and  $(\ell_p(  \Natural), C(K))$
 satisfy the  linear extension property (\cite[Theorem 3.1]{LindPel} and \cite{JohnZippin}).
\item[(iii)] For every compact metric space $K$, the pair $(X,C(K))$ satisfies the linear extension property whenever $X$ is an Orlicz space with a separable dual \cite{Kalton}.
\item[(iv)] The pair $(X,c_0(\Natural))$ satisfies the linear extension property whenever $X$ is a separable Banach space \cite{Sobczyk}.
 \end{enumerate}
\end{ex}

We shall prove the following useful proposition.
\begin{prop}\label{meanvalue:of:smooth}
Let  $(X, Z)$ be a pair of Banach spaces satisfying the linear extension property and $Y$ a closed subspace of $X$. If $f:Y\to Z$ is a $C^1$ smooth mapping  ($C^1$ smooth and Lipschitz mapping), then $f$ satisfies the mean value condition (mean value condition for a bounded map, respectively) on $Y$.
\end{prop}

\begin{proof}
First, let us give the following  lemma.
\begin{lem} 
Let  $(X, Z)$ be a pair of Banach spaces satisfying the \emph{linear extension property} and $Y$ a closed subspace of $X$. Then there is a constant $\eta\ge 1$ and  there is a continuous map $B:\mathcal{L}(Y,Z)\to \mathcal{L}(X,Z)$ such that $B(f)_{\mid_{Y}}=f$  and  $||B(f)||_{\mathcal{L}(X,Z)}\le \eta ||f||_{\mathcal{L}(Y,Z)}$ for every $f\in \mathcal{L}(Y,Z)$.
\end{lem}
The proof of this lemma follows the lines of the real-valued case \cite[Lemma 2]{Azafrykeener}. Indeed, let us  take $W=\mathcal{L}(X,Z)$, $V= \mathcal{L}(Y,Z)$
 and $T:W\to V$ the bounded and linear map given by the restriction to $Y$, $T(f)=f_{\mid_{Y}}$. By assumption, the map $T$   is onto. Thus,  we  apply the Bartle-Graves's theorem (see \cite[Lemma VII 3.2]{DGZ}) in order to find the map $B$.

Now, if $f:Y\to Z$ is a $C^1$ smooth mapping ($C^1$ smooth and Lipschitz mapping), we consider the mapping $D:Y\to  \mathcal{L}(X,Z)$ defined as $D(y):=B(f'(y))$ for every $y\in Y$. Then,  $f$ satisfies the mean value condition for $D$ (the mean value condition for the bounded map $D$, respectively).
\end{proof}

Now, we can apply Theorems \ref{theorem:extension} and  \ref{theorem:Lipschitz:extension} to obtain the following result on $C^1$ smooth extensions and $C^1$ smooth and Lipschitz extensions to $X$ of $C^1$ smooth mappings  defined on $Y$ whenever  $(X,Z)$ satisfies  property  (*) and the linear extension property.

\begin{cor}
Let $(X,Z)$ be any of the following  pairs of Banach spaces:
\begin{enumerate}
\item[(i)] $(L_p,L_2)$,   $2<p<\infty$,
\item[(ii)]  $(c_0(\Gamma), C(K))$,  $\Gamma$ is  a non-empty set and $K$  is a compact metric  space,
\item[(iii)] $(\ell_p(\Natural),C(K))$,  $1<p<\infty$ and $K$ is a compact metric space,
\item[(iv)] $(X,C(K))$,   $X$ is an Orlicz space with separable dual and $K$ is a compact metric space,
\item[(v)] $(X,c_0(\Natural))$, $X$ with separable dual,
\item[(vi)] $(X,\Real)$, such that there is a set $\Gamma\neq \emptyset$  and there is a bi-Lipschitz homeomorphism $\varphi:X\to c_0(\Gamma)$ with $C^1$ smooth coordinate functions (for instance, when  $X^*$
is separable).
\end{enumerate}
 Let $Y$ be a closed subspace of  $X$.  Then, every $C^1$ smooth  mapping $f:Y\rightarrow Z$  has a $C^1$ smooth extension to $X$.

Moreover, there is $C\ge 1$, which depends only on $X$ and $Z$, such that every Lipschitz and $C^1$ smooth  mapping  $f:Y \to Z$ has a Lipschitz and $C^1$ smooth  extension $F:X\to Z$ to  $X$ with $\Lip(F)\le C\Lip(f)$.
\end{cor}


Let us now consider the following definition.
\begin{defn}
Let $X,Z$ be Banach spaces and $Y$ a closed subspace of $X$. We say that the pair $(Y,Z)$ has the \textbf{linear $X$-extension property} if there is $\lambda\ge 1$, which depends on $X$, $Y$ and $Z$, such that for every  bounded and linear map $T: Y\to Z$ there is a bounded and linear extension $\widetilde{T}:X\to Z$ with $||\widetilde{T}||_{\LsXZ}\le \lambda ||T||_{\LsYZ}$.
\end{defn}

By Theorem \ref{theorem:extension}, Theorem \ref{theorem:Lipschitz:extension} and a slight modification of Proposition \ref{meanvalue:of:smooth}, we obtain the following corollary.

\begin{cor}
Let $(X, Z)$ be a pair of Banach spaces with  property (*). Let $Y$ be a closed subspace of $X$ such that the pair $(Y,Z)$ has the  linear $X$-extension property. Then, every $C^1$ smooth  mapping  $f:Y\rightarrow Z$ has a $C^1$ smooth  extension to $X$.

Moreover, there is $C\ge 1$, which depends on $X$, $Y$ and $Z$, such that every Lipschitz and $C^1$ smooth  mapping  $f:Y \to Z$ has a Lipschitz and $C^1$ smooth  extension $F:X\to Z$ to  $X$ with $\Lip(F)\le C\Lip(f)$.
\end{cor}


We conclude this note with some considerations on  extension morphisms of $C^1$
smooth mappings. Let $X$ and $Z$ be Banach spaces and consider the Banach space
\begin{equation*}
C_L^1(X,Z):=\{f:X\toÊZ\ :\,  f \text{ is } C^1 \text{ smooth and Lipschitz}\},
\end{equation*}
with the norm $||f||_{C_L^1}:=||f(0)||+\Lip(f)$. We write  $C_L^1(X):=C_L^1(X,\Real)$.
\begin{defn}
Let $X$ and $Z$ be Banach spaces and  $Y$  a closed subspace of  $X$.
We say that a bounded and linear mapping $T:C_L^1(Y,Z) \to C_L^1(X,Z)$ ($T:Y^* \rightarrow X^*$) is an extension morphism whenever $T(f)_{\mid_Y}=f$ for every $f\in C_L^1(Y,Z)$ (for every $f\in Y^*$, respectively).
\end{defn}

\begin{lem}\label{morphism}
Let $X$ be a Banach space and $Y$ a closed subspace of $X$. If there exists an extension morphism $T:C_L^1(Y)\to C_L^1(X)$, then there  exists an extension morphism $S:Y^* \to X^*$.
\end{lem}

\begin{proof}
 Let  $T: C_L^1(Y)\to C_L^1(X)$ be an extension morphism and define
 $D:C_L^1(X)\to X^*$ as $D(f)=f'(0)$ for every $f\in C_L^1(X)$. The mapping $D$ is  linear, bounded and $||D||\le 1$.  Thus, $D\circ T:C_L^1(Y)\to X^*$ is linear and bounded. Also,  $(D\circ T(f))_{\mid_Y}=(T(f)'(0))_{\mid_Y}=f'(0)\in \mathcal L(Y,Z)$. Now, let us take the restriction $S:=D\circ  T_{\mid_{Y^*}}:Y^*\to X^*$, which is linear, bounded,
\begin{align*}
 S(\varphi)_{\mid_Y} & =(T(\varphi))'(0)_{\mid_Y}=\varphi'(0)=\varphi, \text{ and} \\
    ||S(\varphi)||_{X^*} & = ||D\circ T(\varphi)||_{X^*} \le ||T(\varphi)||_{C_L^1}\le ||T||\,||\varphi||_{C_L^1},
\end{align*}
for every $\varphi \in Y^*$.
\end{proof}

The above lemma and the results given by H. Fakhoury in \cite{Fakhoury} provide the
following characterizations.

\begin{prop}\label{equivalencias}
Let $X$ and $Z$ be Banach spaces. The following statements are equivalent:
\begin{enumerate}
\item[(i)] There is a constant $M> 0$  such that  for every closed subspace $Y\subset X$, there exists an   extension morphism  $P: C_L^1(Y,Z)\to C_L^1(X,Z)$ with $||P||\le M$.
\item[(ii)] There is a constant $M> 0$  such that  for every closed subspace $Y\subset X$, there exists an   extension morphism  $T: C_L^1(Y)\to C_L^1(X)$ with $||T||\le M$.
\item[(iii)] There is a constant $M> 0$  such that  for every closed subspace $Y\subset X$, there exists an extension morphism  $S: Y^* \to X^*$ with $||S||\le M$.
\item[(iv)]  $X$ is isomorphic to a Hilbert space.
\end{enumerate}
\end{prop}

\begin{proof}
The equivalence between $(iii)$ and $(iv)$ was established in \cite[Th\'eor\`eme 3.7]{Fakhoury}.

$(i)\Rightarrow(ii)$ Let us take $z\in S_Z$ and $\varphi\in S_{Z^*}$ (where $ S_Z$ and $S_{Z^*}$ denote the unit spheres of $Z$ and $Z^*$, respectively)
 such that $\varphi(z)=1$. Let $Y$ be a closed subspace of $X$ and $P: C_L^1(Y,Z)\to C_L^1(X,Z)$ an  extension morphism with $||P|| \le M$. For every $f\in C_L^1(Y)$,   let us consider  the Lipschitz and $C^1$ smooth mapping $f_z:  Y \to   Z$ defined as $f_z(y)=zf(y)$ for every $y\in Y$.  Let us define $R_f:=P(f_z)\in C_L^1(X,Z) $. Then, the mapping $T: C_L^1(Y)\to C_L^1(X)$ defined as
 $T(f)=\varphi \circ R_f$,  where  $f\in C_L^1(Y)$  is a linear extension mapping with $||T||\le M$.

$(ii)\Rightarrow (iii)$ is given by Lemma \ref{morphism} and $(iv)\Rightarrow(i)$ follows from the
result  that every closed subspace of a Hilbert space $H$ is complemented in $H$ \cite{LindTza}.
\end{proof}

Recall that $X$ is a $\mathscr{P}_\lambda$-space if $X$ is a complemented subspace of every Banach superspace $W$ of $X$ (see \cite[p. 95]{Day}).

\begin{cor}
Let $X$ and $Z$ be Banach spaces. The following statements are equivalent:
\begin{enumerate}
\item[(i)] There is a constant $M> 0$  such that  for every Banach superspace $W$ of $X$, there exists an   extension morphism  $P: C_L^1(X,Z)\to C_L^1(W,Z)$ with $||P||\le M$.
\item[(ii)] There is a constant $M> 0$  such that  for every Banach  superspace $W$ of $X$, there exists an   extension morphism  $T: C_L^1(X)\to C_L^1(W)$ with $||T||\le M$.
\item[(iii)] There is a constant $M> 0$  such that  for every Banach superspace $W$ of $X$, there exists an extension morphism  $S: X^* \to W^*$ with $||S||\le M$.
\item[(iv)]  $X$ is a $\mathscr{P}_\lambda$-space.
\end{enumerate}
\end{cor}

\begin{proof}
The equivalence between $(iii)$ and $(iv)$ was established in \cite[Corollaire 3.3]{Fakhoury}. The rest of the proof is similar to that of Proposition  \ref{equivalencias}.
\end{proof}


\section*{Acknowledgments}

Luis S\'anchez-Gonz\'alez conducted part of this research while at the Institut de Math\'ematiques de Bordeaux. This author is indebted to the members of the Institut  and very especially to Robert Deville for their kind hospitality and for many useful discussions. The authors also wish to thank Jes\'us M.F. Castillo, Ricardo Garc\'ia and Jes\'us Su\'arez for several valuable discussions concerning the results of this paper.
  


\begin{thebibliography}{998}
\bibitem{AlbiacKalton} F. Albiac and N.J. Kalton, \emph{Topics in Banach space theory}, Graduate Texts in Mathematics, vol. \textbf{233}, Springer, New York (2006).
\bibitem{Azafrykeener} D. Azagra, R. Fry and L. Keener, \emph{Smooth extension of functions on separable Banach spaces}, Math. Ann. \textbf{347}  (2) (2010), 285-297.
\bibitem{Ball} K. Ball, \emph{Markov chains, Riesz transforms and Lipschitz maps}, Geom. Funct. Anal. \textbf{2} (2) (1992), 137-172.
\bibitem{BenyLind} Y. Benyamini and J. Lindenstrauss, \emph{Geometric Nonlinear Functional Analysis, Volumen I}, Amer. Math. Soc. Colloq. Publ. vol. \textbf{48}, Amer. Math. Soc., Providence, RI  (2000).
\bibitem{Day} M. Day, \emph{Normed Linear Spaces}, Springer, Berlin (1963).
\bibitem{DGZ} R. Deville, G. Godefroy and V. Zizler, \emph{Smoothness and renormings in Banach spaces}, Pitman Monographs and Surveys in Pure and Applied Mathematics, vol. \textbf{64} (1993).
\bibitem{Dugundji} J. Dugundji, Topology, Allyn and Bacon, Inc., Boston, Mass. (1966).
\bibitem{Fakhoury} H. Fakhoury, \emph{S\'elections lin\'eaires associ\'ees au th\'eor\`eme de Hahn-Banach}, J. Funct. Anal.  \textbf{11}  (1972), 436-452.
\bibitem{Fry} R. Fry, \emph{Corrigendum to :``Approximation by $C^p$-smooth, Lipschitz functions on Banach spaces, J. Math. Anal.  Appl. 315 (2006), 599-605''}, J. Math. Anal.  Appl. \textbf{348} (1)  (2008),  571.
\bibitem{HJc0} P. H\'ajek and M. Johanis, \emph{Uniformly G\^ateaux smooth approximations on $c_0(\Gamma)$}, J. Math. Anal. Appl. \textbf{350} (2) (2009), 623-629.
\bibitem{HJ} P. H\'ajek and M. Johanis, \emph{Smooth approximations}, J. Funct. Anal. \textbf{259} (3) (2010), 561-582.
\bibitem{MarLuis} M. Jim\'enez-Sevilla and L. S\'anchez-Gonz\'alez, \emph{Smooth extension of functions on a certain class of non-separable Banach spaces}, J. Math. Anal. Appl. \textbf{378} (1) (2011),  173-183.
\bibitem{JohLindSche} W.B. Johnson, J. Lindenstrauss and G. Schechtman, \emph{Extensions of Lipschitz maps into Banach spaces}, Israel J. Math.  \textbf{54} (2) (1986), 129-138.
\bibitem{JohnZippin} W.B. Johnson and M. Zippin, \emph{Extension of operators from subspaces of {$c_0(\Gamma)$}  into {$C(K)$} spaces}, Proc. Amer. Math. Soc. \textbf{107} (3) (1989), 751-754.
\bibitem{Kalton} N.J. Kalton, \emph{Extension of linear operators and Lipschitz maps into $C(K)$-spaces}, New York J. Math. \textbf{13}  (2007), 317-381.
\bibitem{Kirs} M.D. Kirszbraun, \emph{\"Uber die zusammenziehende und Lipschitzche Transformationen}, Fund. Math. \textbf{22} (1934), 77-108.
\bibitem{LindPel} J. Lindenstrauss and A. Pe{\l}czy{\'n}ski, \emph{Contributions to the theory of the classical Banach spaces}, J. Funct. Anal. \textbf{8} (1971), 225-249.
\bibitem{LindTza} J. Lindenstrauss and L. Tzafriri, \emph{On the complemented subspaces problem},  Israel J. Math. \textbf{9} (1971), 263-269.
\bibitem{Maurey} B. Maurey, \emph{Un th\'eor\`eme de prolongement}, C. R. Acad. Sci. Paris S\'er. A  \textbf{279}  (1971), 329-332.
\bibitem{Naor} A. Naor, \emph{A phase transition phenomenon between the isometric and isomorphic extension problems for H\"{o}lder functions between $L_p$ spaces}, Mathematika \textbf{48} (2001), 253-271.
\bibitem{Rudin} M.E. Rudin, \emph{A new proof that metric spaces are paracompact}, Proc. Amer. Math. Soc. \textbf{20} (2) (1969), 603.
\bibitem{Sobczyk} A. Sobczyk, \emph{Projection of the space $(m)$ on its subspace $(c_0)$}, Bull. Amer. Math. Soc. \textbf{47}  (1941), 938-947.
\bibitem{Tsar} I.G.  Tsar'kov, \emph{Extension of Hilbert-valued Lipschitz mappings},
Moscow Univ. Math. Bull.  \textbf{54} (6) (1999), 7-14.
\end{thebibliography}
\end{document}